\documentclass[11pt]{amsart}
\usepackage{amsfonts}
\usepackage{color}
\usepackage{amsmath}
\usepackage{amssymb}
\usepackage{graphicx}

\providecommand{\U}[1]{\protect\rule{.1in}{.1in}}
\providecommand{\U}[1]{\protect\rule{.1in}{.1in}}
\newtheorem{theorem}{Theorem}[section]

\newtheorem{condition}[theorem]{Condition}

\newtheorem{definition}[theorem]{Definition}
\newtheorem{example}[theorem]{Example}

\newtheorem{lemma}[theorem]{Lemma}

\newtheorem{proposition}[theorem]{Proposition}
\newtheorem{remark}[theorem]{Remark}

\newcommand{\uemptyset }{{\boldsymbol{\emptyset}}}

\begin{document}
\title[Crystal isomorphisms and wall crossing maps]{Crystal isomorphisms and
wall crossing maps for rational Cherednik algebras }
\author{Nicolas Jacon and C\'edric Lecouvey}
\date{March 2016}

\begin{abstract}
We show that the wall crossing bijections between simples of the category $%
\mathcal{O}$ of the rational Cherednik algebras reduce to particular crystal
isomorphisms which can be computed by a simple combinatorial procedure on
multipartitions of fixed rank.
\end{abstract}

\maketitle

\section{Introduction}

The rational Cherednik algebra associated to a complex reflection group $W$
was introduced by Etingof and Ginzburg in \cite{EtiGin} as a particular
symplectic reflection algebra. In this paper, we will focus on the case
where $W$ is the complex reflection group $G(l,1,n):=(\mathbb{Z}/l\mathbb{Z}%
)^{n}\rtimes \mathfrak{S}_{n}$. The rational Cherednik algebra then depends
on the choice of a certain parameter $s$ in a $(l+1)$-dimensional $\mathbb{C}
$-vector space. There is a distinguished category of modules over these
algebras, the category $\mathcal{O}$, which may be constructed in the same
spirit as the BGG category for a reductive Lie algebra. This category has
been intensively studied during the last decade because of its interesting
structure and also its connection with other important mathematical objects
such as cyclotomic Hecke algebras or cyclotomic $q$-Schur algebras.

As the simple modules of the associated complex reflection groups, the
simple modules in this category are labelled by the set of multipartitions.\
A natural and important question is then to understand how are related the
set of simple modules in the category $\mathcal{O}$ for different choices of
the parameter $s$. In \cite{Lo}, Losev has defined a collection of
hyperplanes in the space of parameters called \textquotedblleft essential
walls\textquotedblright . Then, he has shown the existence of a perverse and
derived equivalence between the categories $\mathcal{O}$ associated to
distinct parameters $s$ and $s^{\prime }$ separated by a single wall. In
particular, this equivalence induces a bijection between the simples of
these categories. These bijections, called \textquotedblleft wall crossing
maps\textquotedblright , are of great interest because they commute with
both actions of the Heisenberg algebra and the affine type $A$ algebra $%
\mathfrak{g}$. In particular, they can be used to compute the support of the
simple modules in the category $\mathcal{O}$ and to obtain a classification
of the finite dimensional irreducible representations. The goal of the
present paper is to show that, despite their very abstract definition, there
is a simple combinatorial procedure to compute them.

Because of the above properties, these bijections can also be interpreted as
crystal isomorphisms for certain integrable $\mathfrak{g}$-modules: the Fock
spaces. On the other hand, in \cite{JL}, the authors of the present paper
have described distinguished crystal isomorphisms between such Fock spaces
and presented a simple procedure on multipartitions of fixed rank to compute
them. The aim of this paper is to show how these two isomorphisms are
related. To do this, we first review the Uglov $\mathfrak{g}$-module
structure of the Fock space and explain how it can be extended. We then
construct and explicitly describe crystal isomorphisms corresponding to this
generalized structure. The last section explores the connections of these
isomorphisms with Losev's wall crossing bijections. We explain how the
crystals appearing in the context of Cherednik algebras are related with the
usual crystals of Fock spaces (in the sense of Uglov). We obtain in
particular an easy criterion to decide whether an $l$-partition is a highest
weight vertex for this structure. Finally, the main result gives a simple
combinatorial procedure on partitions of fixed rank for computing an
arbitrary wall crossing bijection without referring to any crystal structure

\vspace{0.5cm}

\noindent\textbf{Acknowledgement.} The authors thank Ivan Losev for useful
discussions on his work \cite{Lo,Losurvey}.\newline
Both Authors are supported by ``Projet ANR ACORT : ANR-12-JS01-0003".

\section{Colored graphs and Fock spaces of JMMO type}

We first describe structures of colored oriented graphs on the sets of $l$%
-partitions (that we define below), then focus on a particular case which is
connected to the action of the quantum affine group ${\mathcal{U}_{q}(%
\widehat{\mathfrak{sl}_{e}})}$.

\subsection{Combinatorics of $l$-partitions}

\begin{definition}
\label{def1} Let $n\in\mathbb{N}$ and $l\in\mathbb{N}$:

\begin{itemize}
\item A \textit{partition} $\lambda$ of rank $n$ is a sequence 
\begin{equation*}
(\lambda_{1},\ldots,\lambda_{r}), 
\end{equation*}
of decreasing non negative integers such that 
\begin{equation*}
\sum_{1\leq i\leq r}\lambda_{i}=n. 
\end{equation*}
By convention, we identify two partitions which differ by parts equal to $0$.

\item A \textit{$l$-partition} (or multipartition) ${\boldsymbol{\lambda}}$
of rank $n$ is an $l$-tuple of partitions $(\lambda^{1},\ldots,\lambda^{l})$
such that the sum of the rank of the partitions $\lambda^{i}$ for $1\leq
i\leq l$ is $n$. We denote by $\Pi^{l}(n)$ the set of all $l$-partitions of
rank $n$ and if ${\boldsymbol{\lambda}}\in\Pi^{l}(n)$, we sometimes write ${%
\boldsymbol{\lambda}}\vdash_{l}n$. The empty $l$-partition (which is the $l$%
-tuple of empty partitions) is denoted by $\uemptyset$.
\end{itemize}
\end{definition}

Let $\kappa\in\mathbb{Q}_{+}$ and $\mathbf{s}=(s_{1},s_{2},\ldots,s_{l})\in%
\mathbb{Q}^{l}$. We set $s:=(\kappa,\mathbf{s})$.

One can associate to each ${\boldsymbol{\lambda}}\vdash_{l}n$ its \textit{%
Young diagram}: 
\begin{equation*}
\lbrack{\boldsymbol{\lambda}}]=\{(a,b,c)\ a\geq1,\ 1\leq c\leq l,1\leq
b\leq\lambda_{a}^{c}\}. 
\end{equation*}
This diagram will be sometimes identify with the $l$-partition itself. We
define the $s$-\textit{content} of a node $\gamma=(a,b,c)\in\lbrack {%
\boldsymbol{\lambda}}]$ as follows: 
\begin{equation*}
\text{cont}(\gamma)=b-a+s_{c}\in\mathbb{Q}, 
\end{equation*}
and the \textit{residue }of $\gamma$ is by definition the content of the
node in $\mathbb{Q}/\kappa^{-1}\mathbb{Z}$.

\begin{definition}
\label{DefIs}Let $I_{s}$ be the subset of $\mathbb{Q}/\kappa^{-1}\mathbb{Z}$
formed by the classes $x+s_{j}+\kappa^{-1}\mathbb{Z},j=1,\ldots,l$ where $%
x\in\mathbb{Z}$ and $j\in\{1,\ldots,l\}$.
\end{definition}

We say that $\gamma$ is a \textit{$z$}-node of ${\boldsymbol{\lambda}}$ when 
$\mathrm{res}(\gamma)=z+\kappa^{-1}\mathbb{Z}.$ Finally, we say that $\gamma$
is \textit{removable} when $\gamma=(a,b,c)\in{\boldsymbol{\lambda}}$ and ${%
\boldsymbol{\lambda}}\backslash\{\gamma\}$ is an $l$-partition. Similarly $%
\gamma$ is \textit{addable} when $\gamma=(a,b,c)\notin{\boldsymbol{\lambda}}$
and ${\boldsymbol{\lambda}}\cup\{\gamma\}$ is an $l$-partition.

\bigskip

Fix $z\in I_{s}$ . We assume that we have a total order $\leq$ on the set of 
$z$-nodes of an arbitrary $l$-partition. We then define two operators
depending on $z$ as follows. We consider the set of addable and removable $z$%
-nodes of our $l$-partition. Let $w_{z}({\boldsymbol{\lambda}})$ be the word
obtained first by writing the addable and removable $z$-nodes of ${{{%
\boldsymbol{\lambda}}}}$ in {increasing} order with respect to $\leq$ next
by encoding each addable $z$-node by the letter $A$ and each removable $i$%
-node by the letter $R$.\ Write $\widetilde{w}_{z}({\boldsymbol{\lambda}}%
)=A^{p}R^{q}$ for the word derived from $w_{z}$ by deleting as many subwords
of type $RA$ as possible. The word $w_{z}({\boldsymbol{\lambda}})$ is called
the {$z$-word} of ${\boldsymbol{\lambda}}$ and $\widetilde{w}_{z}({%
\boldsymbol{\lambda}})$ the {reduced $z$-word} of ${\boldsymbol{\lambda}}$.
The addable $z$-nodes in $\widetilde{w}_{z}({\boldsymbol{\lambda}})$ are
called the \textit{normal addable $z$-nodes}. The removable $z$-nodes in $%
\widetilde{w}_{z}({\boldsymbol{\lambda}})$ are called the \textit{normal
removable $z$-nodes}. If $p>0,$ let $\gamma$ be the rightmost addable $z$%
-node in $\widetilde{w}_{z}$. The node $\gamma$ is called the \textit{good
addable $z$-node}. If $q>0$, the leftmost removable $i$-node in $\widetilde{w%
}_{z}$ is called the \textit{good removable $z$-node}.

We then define $\widetilde{e}_{z}^{\leq}{\boldsymbol{\mu}}={\boldsymbol{%
\lambda}}$ and $\widetilde{f}_{z}^{\leq}{\boldsymbol{\lambda}}={\boldsymbol{%
\mu}}$ if and only if ${{{\boldsymbol{\mu}}}}$ is obtained from ${%
\boldsymbol{\lambda}}$ by adding to ${{{\boldsymbol{\lambda}}}}$ a good
addable $z$-node, or equivalently, ${\boldsymbol{\lambda}}$ is obtained from 
${\boldsymbol{\mu}}$ by removing a good removable $z$-node. If ${\boldsymbol{%
\mu}}$ as no good removable $z$-node then we set $\widetilde {e}_{z}^{\leq}{%
\boldsymbol{\mu}}=0$ and if ${\boldsymbol{\lambda}}$ has no good addable $z$%
-node we set $\widetilde{f}_{z}^{\leq}{\boldsymbol{\lambda}}=0$.

\subsection{Extended JMMO structure associated to $s$}

\label{Gr} We can define from $s$ and $\leq$ a colored oriented graph $%
\mathcal{G}_{s,\leq}$ as follows:

\begin{itemize}
\item vertices : the $l$-partitions ${\boldsymbol{\lambda}}\vdash_{l} n$
with $n\in\mathbb{Z}_{\geq0}$

\item the arrows are colored by elements in $I_{s}$ and we have: ${%
\boldsymbol{\lambda}}\overset{i}{\rightarrow}{\boldsymbol{\mu}}$ for $i\in%
\mathbb{Q}/\kappa^{-1}\mathbb{Z}$ if and only if $\widetilde{e}_{i}^{\leq}{%
\boldsymbol{\mu}}={\boldsymbol{\lambda}}$, or equivalently $\widetilde{f}%
_{i}^{\leq}{\boldsymbol{\lambda}}={\boldsymbol{\mu}}$.
\end{itemize}

The $l$-partitions such that $\widetilde{e}_{z}{\boldsymbol{\mu}}=0$ for all 
$z$ will be called \textit{highest weight vertices}. The set $I_{s}$ of
elements in $\mathbb{Q}/\kappa^{-1}\mathbb{Z}$ coloring the arrows is called
the indexing set of the graph. Observe the graph $\mathcal{G}_{s,\leq}$ is
not an affine $A_{e-1}^{(1)}$-crystal graph in general notably because its
indexing set can be distinct from $\mathbb{Z}/e\mathbb{Z}$.

\label{Giso}

Finally two graphs $\mathcal{G}_{s_{1},\leq_{1}}$ and $\mathcal{G}%
_{s_{1},\leq_{2}}$ on $l$-partitions with indexing sets $I_{1}$ and $I_{2}$
are isomorphic if there exists a bijection 
\begin{equation}
\Psi:\Pi^{l}(n)\rightarrow\Pi^{l}(n),   \label{Psi}
\end{equation}
and a bijection: 
\begin{equation}
\psi:I_{1}\rightarrow I_{2},   \label{psi}
\end{equation}
such that

\begin{itemize}
\item ${\boldsymbol{\lambda}}$ is a highest weight vertex in $\mathcal{G}%
_{s_{1},\leq_{1}}$ if and only if $\Psi({\boldsymbol{\lambda}})$ is a
highest weight vertex in $\mathcal{G}_{s_{2},\leq_{2}}$,

\item we have an arrow ${\boldsymbol{\lambda}}\overset{i}{\rightarrow }{%
\boldsymbol{\mu}}$ in $\mathcal{G}_{s_{1},\leq_{1}}$ if and only if we have
an arrow $\Psi({\boldsymbol{\lambda}})\overset{\psi(i)}{\rightarrow}\Psi({%
\boldsymbol{\mu}})$ in $\mathcal{G}_{s_{2},\leq_{2}}$.
\end{itemize}

In addition, if $\Psi$ is the identity, we will say that the two graphs are
equivalent (in particular, the graph structures coincide up to their
coloring). We will see that, for a good choice of the orders $\leq_{1}$ and $%
\leq_{2}$, the graphs $\mathcal{G}_{s_{1},\leq_{1}}$ and $\mathcal{G}%
_{s_{2},\leq_{2}}$ have the structure of a Kashiwara crystal graph. In this
case, a crystal isomorphism between $\mathcal{G}_{s_{1},\leq_{1}}$ and $%
\mathcal{G}_{s_{2},\leq_{2}}$ is a graph isomorphism such that $\psi=id$. In
particular, each crystal isomorphism yields a graph isomorphism but the
converse is false in general.

\subsection{Extended JMMO Fock space structure}

\label{act1}

We now define a Fock space structure that, as far as we know, first appeared
in the work of Gerber \cite{Ger}. This structure generalizes the ones
defined by Uglov \cite{uglov} and Jimbo-Misra-Miwa-Okado (JMMO) \cite{jmmo}.

\begin{condition}
We assume in this paragraph that $\mathbf{s}=(s_{1},\ldots,s_{l})\in \mathbb{%
Z}^{l}$ and $\kappa=1/e$ where $e\in\mathbb{N}_{>1}\sqcup\{\infty\}$.
\end{condition}

Let $\mathfrak{g}_{e}:={\mathcal{U}_{q}(\widehat{\mathfrak{sl}_{e}})}$ be
the quantum group of affine type $A_{e-1}^{(1)}$ if $e$ is finite, otherwise
we set $\mathfrak{g}_{\infty}:={\mathcal{U}_{q}({\mathfrak{sl}_{\infty}})}$.
The associative $\mathbb{Q}(q)$-algebra $\mathfrak{g}_{e}$ has generators $%
e_{i},f_{i},t_{i},t_{i}^{-1}$ (for $i=0,...,e-1$) and $\partial$.\ We refer
the reader to \cite[\S 2.1]{uglov} for the relations they satisfy (we do not
use them in the sequel.) We denote by ${\mathcal{U}_{q}^{\prime}(\widehat{%
\mathfrak{sl}_{e}})}$ the subalgebra generated by $%
e_{i},f_{i},t_{i},t_{i}^{-1}$ (for $i=0,...,e-1$). We write $%
\Lambda_{i},i=0,...,e-1$ for the fundamental weights of $\mathfrak{g}_{e}$.

\label{act2}The \textit{Fock space} $\mathcal{F}$ is the $\mathbb{Q}(q)$%
-vector space defined as follows: 
\begin{equation*}
\mathcal{F}=\bigoplus_{n\in\mathbb{Z}_{\geq0}}\bigoplus_{{\boldsymbol{%
\lambda }}\vdash_{l}n}\mathbb{Q}(q){\boldsymbol{\lambda}}. 
\end{equation*}
Set ${I}=\mathbb{Z}$ if $e=\infty$ and $I=\mathbb{Z}/e\mathbb{Z}$ otherwise.

For any $e\in\mathbb{N}_{>1}\sqcup\{\infty\}$, there is an action of $%
\mathfrak{g}_{e}$ on the Fock space $\mathcal{F}$. This action can be
regarded as a generalization of Uglov's construction \cite[\S 2.1]{uglov}.
It is defined by using an order on the nodes on the $l$-partitions with the
same residue modulo $e$. This order depends on the choice of an $l$-tuple of
rational numbers $\mathbf{m}=(m_{1},\ldots,m_{l})$. If $e$ is finite, for $%
1\leq i,j\leq l$ and $N\in\mathbb{Z}$, let us define 
\begin{equation*}
\mathfrak{m}_{i,j,N}^{\mathbf{s},e}:=\{(m_{1},\ldots,m_{l})\in\mathbb{Q}%
^{l}\ |s_{i}-m_{i}-(s_{j}-m_{j})=N.e\}. 
\end{equation*}
If $e=\infty$, we define for $1\leq i,j\leq l$ 
\begin{equation*}
\mathfrak{m}_{i,j}^{\mathbf{s},\infty}:=\{(m_{1},\ldots,m_{l})\in \mathbb{Q}%
^{l}\ |s_{i}-m_{i}-(s_{j}-m_{j})=0\}. 
\end{equation*}
Let $\mathfrak{M}^{\mathbf{s},e}$ be the union of the hyperplanes $\mathfrak{%
m}_{i,j,N}^{\mathbf{s},e}$ for all $1\leq i,j\leq l$ and $N\in\mathbb{Z}$ if 
$e$ is finite and $\mathfrak{m}_{i,j}^{\mathbf{s},\infty}$ for all $1\leq
i,j\leq l$ if $e=\infty$. Now consider $\mathbf{m}\notin\mathfrak{M}^{%
\mathbf{s},e}$. Let $\gamma$, $\gamma^{\prime}$ be two removable or addable $%
i$-nodes of ${\boldsymbol{\lambda}}$ for $i\in I_{s}$. We denote 
\begin{equation}
\gamma\preceq_{\mathbf{m}}\gamma^{\prime}\overset{\text{def}}{%
\Longleftrightarrow}b-a+m_{c}<b^{\prime}-a^{\prime}+m_{c^{\prime}}. 
\label{order_m}
\end{equation}
Thanks to our assumption $\mathbf{m}\notin\mathfrak{M}^{\mathbf{s},e}$, it
is easy to verify that the above definition indeed defines a total order on
the set of $i$-nodes of any $l$-partition. This thus yields a graph $%
\mathcal{G}_{e,\mathbf{m},\mathbf{s}}$ as in \S \ref{Gr}. It was also proved
in \cite{Ger} that one can mimic Uglov's Fock space construction and define
a $\mathfrak{g}_{e}$-action on the Fock space $\mathcal{F}$ from any order $%
\preceq_{\mathbf{m}}$. This gives an integrable $\mathfrak{g}_{e}$-module
that we denote by $\mathcal{F}_{e,\mathbf{m},\mathbf{s}}$. The submodule
generated by the empty $l$-partition is then an irreducible highest weight
module of weight $\Lambda_{\mathbf{s}}=\Lambda_{s_{1}}+\cdots+\Lambda_{s_{l}}
$.

\begin{remark}
\label{reegalite} Note that the Fock space $\mathcal{F}_{e,\mathbf{m},%
\mathbf{s}}$ depends on the choice of $\mathbf{m}$ (because of the order $%
\preceq_{\mathbf{m}}$) and on the choice of $\mathbf{s}$ modulo $e$ (because
of the definition of the residue of a node). Also in the case where $\mathbf{%
s}=(s_{1},\ldots,s_{l})$ and $\mathbf{s}^{\prime}=(s_{1}^{\prime
},\ldots,s_{l}^{\prime})$ satisfy $s_{j}\equiv s_{j}^{\prime}\ (\text{mod }e)
$ for $j=1,\ldots,l$, the associated Fock spaces can be identified and we
can write $\mathcal{F}_{e,\mathbf{m},\mathbf{s}}=\mathcal{F}_{e,\mathbf{m},%
\mathbf{s}^{\prime}}$.
\end{remark}

\begin{remark}
The inverse order $\preceq_{\mathbf{m}}^{-}$ of $\preceq_{\mathbf{m}}$ also
yields the structure of an integrable $\mathfrak{g}_{e}$-module on $\mathcal{%
F}$\ we denote by $\mathcal{F}_{e,\mathbf{m},\mathbf{s}}^{-}$.
\end{remark}

\subsection{Relations with JMMO Fock space structure (1)}

\label{delta}

The operators $\widetilde{e}_{z}^{\preceq_{\mathbf{m}}}$ and $\widetilde {f}%
_{z}^{\preceq_{\mathbf{m}}}$ defined from the order $\preceq_{\mathbf{m}}$
as in \S \ref{Gr} coincide in fact with the Kashiwara crystal operators and $%
\mathcal{G}_{e,\mathbf{m},\mathbf{s}}$ or $\mathcal{G}_{\infty,\mathbf{m},%
\mathbf{s}}$ are the crystal graphs corresponding to the $\mathfrak{g}_{e}$%
-module structure on our Fock space. To recover the crystal structure used
by Uglov (or the crystal structure introduced by JMMO), it suffices to
choose $\mathbf{m}$ such that 
\begin{equation*}
m_{c}=s_{c}+\delta_{c},\ c=1,\ldots,l, 
\end{equation*}
where $e>\delta_{1}>\cdots>\delta_{l}\geq0$. We will simply denote by $%
\mathcal{G}_{e,\mathbf{s}}$ this JMMO structure.

Now let us consider $\mathbf{m}^{\prime}\notin\mathfrak{M}^{\mathbf{s},e}$.
For $c=1,\ldots,l$, define $\delta_{c}^{\prime}$ to be the unique element of 
$\{0,1\ldots,e-1\}$ which is equivalent to $m_{c}^{\prime}-s_{c}$ modulo $e$%
. Thus there exists $(s_{1}^{\prime},\ldots s_{l}^{\prime})\in\mathbb{Z}^{l}$
such that $s_{j}\equiv s_{j}^{\prime}(\text{mod }e)$ and 
\begin{equation*}
m_{c}^{\prime}=s_{c}^{\prime}+\delta_{c}^{\prime},\ c=1,\ldots,l. 
\end{equation*}
We have 
\begin{equation*}
e>\delta_{\sigma(1)}^{\prime}>\cdots>\delta_{\sigma(l)}^{\prime}\geq0, 
\end{equation*}
for a permutation $\sigma\in\mathfrak{S}_{l}$. Then the map%
\begin{equation}
\left\{ 
\begin{array}{c}
\mathcal{F}_{e,\mathbf{m}^{\prime},\mathbf{s}}=\mathcal{F}_{e,\mathbf{m}%
^{\prime},\mathbf{s}^{\prime}}\overset{\sigma}{\rightarrow}\mathcal{F}%
_{e,\sigma(\mathbf{m}^{\prime}),\sigma(\mathbf{s})} \\ 
(\lambda^{1},\ldots,\lambda^{l})\mapsto(\lambda^{\sigma(1)},\ldots
,\lambda^{\sigma(l)})%
\end{array}
\right.   \label{iso_perm}
\end{equation}
is an isomorphism of $\mathfrak{g}_{e}$-modules. It also defines a crystal
isomorphism between the crystal $\mathcal{G}_{e,\mathbf{m}^{\prime},\mathbf{s%
}}$ and the JMMO crystal $\mathcal{G}_{e,\sigma(\mathbf{s})}$.

This implies that for any $\mathbf{m}_{1}\notin\mathfrak{M}^{\mathbf{s},e}$
and $\mathbf{m}_{2}\notin\mathfrak{M}^{\mathbf{s},e}$, the Fock spaces $%
\mathcal{F}_{e,\mathbf{m}_{1},\mathbf{s}}$ and $\mathcal{F}_{e,\mathbf{m}%
_{2},\mathbf{s}}$ are isomorphic. The crystals $\mathcal{G}_{e,\mathbf{m}%
_{1},\mathbf{s}}$ and $\mathcal{G}_{e,\mathbf{m}_{2},\mathbf{s}}$ are then
also isomorphic as crystals. This means they are isomorphic in the sense of 
\S\ \ref{Giso} with indexing set $\mathbb{Z}/e\mathbb{Z}$ and $\psi=id$. The
modules $\mathcal{F}_{e,\mathbf{m}_{1},\mathbf{s}}$ and $\mathcal{F}_{e,%
\mathbf{m}_{2},\mathbf{s}}$ are reducible in general, so such an isomorphism
is not unique. However, its restriction to the connected component of $%
\mathcal{G}_{e,\mathbf{m}_{1},\mathbf{s}}$ with empty highest weight vertex
yields the connected component of $\mathcal{G}_{e,\mathbf{m}_{2},\mathbf{s}}$
with empty highest weight vertex. In the next sections we shall study
certain \textquotedblleft canonical isomorphisms\textquotedblright\ for
graphs $\mathcal{G}_{e,\mathbf{m},\mathbf{s}}$ defined from a datum $%
s=(\kappa,\boldsymbol{s})$ more general than when $\kappa=\frac{1}{e}$ and $%
\boldsymbol{s}\in\mathbb{Z}^{l}$.


\section{Description of the canonical crystal isomorphisms}

In this section we assume that $s=(\kappa,\boldsymbol{s})$ with $\kappa =%
\frac{1}{e}$ and $\boldsymbol{s}\in\mathbb{Z}^{l}$. Then $\mathcal{G}_{e,%
\mathbf{m},\mathbf{s}}$ is a Kashiwara crystal for any $\mathbf{m}\notin%
\mathfrak{M}^{\mathbf{s},e}$.

\subsection{Crystal isomorphisms}

The hyperplanes $\mathfrak{m}_{i,j,N}^{\mathbf{s},e}$ divide $\mathbb{R}^{l}$
into chambers. We first show that the orders $\preceq_{\mathbf{m}}$ are the
same for all the parameters $\mathbf{m}$ in the same (open) chambers. We
also show that one can restrict to a finite sets of chambers in order to
understand our crystal isomorphisms.

\begin{proposition}
\label{Prop_chamber}Assume that $\mathbf{m}_{1}$ and $\mathbf{m}_{2}$ are
both in the same chamber with respect to the decomposition in \S \ref{act2},
then the orders $\preceq_{\mathbf{m}_{1}}$ and $\preceq_{\mathbf{m}_{2}}$ on
the $i$-nodes of an arbitrary $l$-partition coincide.
\end{proposition}

\begin{proof}
Consider $\gamma=(a,b,c)$ and $\gamma^{\prime}=(a^{\prime},b^{\prime
},c^{\prime})$ two distinct $i$-nodes and assume we have $\gamma \prec_{%
\mathbf{m}_{1}}\gamma^{\prime}$ but $\gamma^{\prime}\prec _{\mathbf{m}%
_{2}}\gamma$. This means that: 
\begin{multline}
b-a+m_{1,c}<b^{\prime}-a^{\prime}+m_{1,c^{\prime}},\ b-a+m_{2,c}>b^{\prime
}-a^{\prime}+m_{2,c^{\prime}}  \notag \\
\text{and }b-a+s_{c}=b^{\prime}-a^{\prime}+s_{c^{\prime}}+ke\text{ with }k\in%
\mathbb{Z}.   \label{rela}
\end{multline}
We get $b^{\prime}-a^{\prime}=b-a+s_{c}-s_{c^{\prime}}-ke$. By replacing $%
b^{\prime}-a^{\prime}$ by its expression in terms of $a$ and $b$ in the
first above inequality we obtain: 
\begin{equation}
(s_{c}-m_{1,c})-(s_{c^{\prime}}-m_{1,c^{\prime}})>ke,   \label{supke}
\end{equation}
whereas the second inequality yields: 
\begin{equation}
(s_{c}-m_{2,c})-(s_{c^{\prime}}-m_{2,c^{\prime}})<ke.   \label{infke}
\end{equation}
This means that $\mathbf{m}_{1}$ and $\mathbf{m}_{2}$ are separated by the
affine hyperplane with equation $(s_{c}-m_{c})-(s_{c^{\prime}}-m_{c^{%
\prime}})=ke$, so we get the desired contradiction.
\end{proof}

\begin{proposition}
\label{asy} Consider a wall $\mathfrak{m}_{i,j,N}^{\mathbf{s},e}$ such that: 
\begin{equation*}
|N.e+(s_{j}-s_{i})|>n, 
\end{equation*}
and pick two parameters $\mathbf{m}_{1}$ and $\mathbf{m}_{2}$ separated by
this wall (and only it) then the order $\preceq_{\mathbf{m}_{1}}$ and $%
\preceq_{\mathbf{m}_{2}}$ are the same on $\Pi^{l}(n)$.
\end{proposition}

\begin{proof}
Consider $\gamma=(a,b,c)$ and $\gamma^{\prime}=(a^{\prime},b^{\prime
},c^{\prime})$ two distinct $i$-nodes such that $\gamma\prec_{\mathbf{m}%
_{1}}\gamma^{\prime}$ but $\gamma^{\prime}\prec_{\mathbf{m}_{2}}\gamma$.
Since we know that $\mathbf{m}_{1}$ and $\mathbf{m}_{2}$ are separated by
the unique wall $\mathfrak{m}_{i,j,N}^{\mathbf{s},e}$. One can assume that $%
c=i$ and $c^{\prime}=j$ and $k=N$. We obtain: 
\begin{equation*}
b-a+s_{i}=b^{\prime}-a^{\prime}+s_{j}+N.e, 
\end{equation*}
but as we have 
\begin{equation*}
|(b-a)-(b^{\prime}-a^{\prime})|\leq n, 
\end{equation*}
this leads to a contradiction.
\end{proof}

\bigskip

The above propositions shows that if we choose $\mathbf{m}_{1}$ and $\mathbf{%
m}_{2}$ such that :

\begin{itemize}
\item these parameters belong to the same chamber with respect to the
decomposition in \S \ref{act2},

\item or satisfy the condition of Proposition \ref{asy},
\end{itemize}

then the associated crystal structures are not simply isomorphic but equal.
To investigate our canonical isomorphims, we thus have a finite set $%
\mathfrak{M}_{n}^{\mathbf{s},e}$ of hyperplanes and chambers to consider
(depending on $n$), namely $\mathfrak{M}_{n}^{\mathbf{s},e}$ is the union of
the $\mathfrak{m}_{i,j,N}^{\mathbf{s},e}$ for $1\leq i<j\leq l$ and $%
\left\vert N.e+(s_{j}-s_{i})\right\vert \leq n$. Indeed, the canonical
isomorphisms we aim to characterize will be equal to the identity if we stay
inside one chamber.\ So it just remains to understand what happens when we
move from one chamber to another.\ This is equivalent to describe the
crystal isomorphisms corresponding to the crossings of the walls $\mathfrak{m%
}_{i,j,N}^{\mathbf{s},e}$ with $1\leq i<j\leq l$ and $\left\vert
N.e+(s_{j}-s_{i})\right\vert \leq n$.

\subsection{A combinatorial procedure}

\label{comb}

We first study the case of $e=\infty$, define and describe our canonical
crystal isomorphisms. The walls are then defined as: 
\begin{equation*}
\mathfrak{m}_{i,j}^{\mathbf{s},\infty}:=\{(m_{1},\ldots,m_{l})\in \mathbb{Q}%
^{l}|s_{i}-m_{i}-(s_{j}-m_{j})=0\}, 
\end{equation*}
for each $1\leq i,j\leq l$. Let us first consider the case where $l=2$. Set $%
\mathbf{s}=(s_{1},s_{2})$. We wish to describe a crystal isomorphism when we
cross a wall: 
\begin{equation*}
\mathfrak{m}_{1,2}^{(s_{1},s_{2}),\infty}:=\{(m_{1},m_{2})\in\mathbb{Q}%
^{2}|s_{1}-m_{1}-(s_{2}-m_{2})=0\}. 
\end{equation*}

Let $(\lambda^{1},\lambda^{2})$ be a bipartition of $n$ and define $%
d\geq\left\vert s_{1}-s_{2}\right\vert $ minimal such that $\lambda
_{d+1-\left\vert s_{1}-s_{2}\right\vert }^{1}=\lambda_{d+1-\left\vert
s_{1}-s_{2}\right\vert }^{2}=0$. To $(\lambda^{1},\lambda^{2})$, we
associate its symbol. This is the two row tableau:{\small 
\begin{align*}
S(\lambda^{1},\lambda^{2}) & =%
\begin{tabular}{|l|l|l|lll}
\hline
$s_{2}-d+\lambda_{d+1}^{2}$ & $\cdots$ & $\cdots$ & $\cdots$ & 
\multicolumn{1}{|l}{$s_{2}-1+\lambda_{2}^{2}$} & \multicolumn{1}{|l|}{$%
s_{2}+\lambda_{1}^{2}$} \\ \hline
$s_{2}-d+\lambda_{d+1+s_{1}-s_{2}}^{1}$ & $\cdots$ & $s_{1}+\lambda_{1}^{1}$
&  &  &  \\ \cline{1-3}
\end{tabular}
\\
\text{when }s_{2} & \geq s_{1} \\
S(\lambda^{1},\lambda^{2}) & =%
\begin{tabular}{|l|l|l|lll}
\cline{1-3}
$s_{1}-d+\lambda_{d+1+s_{2}-s_{1}}^{2}$ & $\cdots$ & $s_{2}+\lambda_{1}^{2}$
&  &  &  \\ \hline
$s_{1}-d+\lambda_{d+1}^{1}$ & $\cdots$ & $\cdots$ & $\cdots$ & 
\multicolumn{1}{|l}{$s_{1}-1+\lambda_{2}^{1}$} & \multicolumn{1}{|l|}{$%
s_{1}+\lambda_{1}^{1}$} \\ \hline
\end{tabular}
\\
\text{when }s_{2} & <s_{1}.
\end{align*}
} We will write $S(\lambda^{1},\lambda^{2})=\binom{L_{2}}{L_{1}}$. By
definition of $d$, the entry in the leftmost column of $S(\lambda^{1},%
\lambda^{2})$ is equal to $s_{2}-d$ (resp. $s_{1}-d$) when $s_{2}\geq s_{1}$
(resp. $s_{2}<s_{1}$). So from $S(\lambda^{1},\lambda^{2})$ it is easy to
obtain $s_{1}$ and $s_{2}$ since $\left\vert s_{1}-s_{2}\right\vert $ is the
difference between the lengths of $L_{1}$ and $L_{2}$. Once we have $%
S(\lambda^{1},\lambda^{2})$ and $(s_{1},s_{2})$, we can recover the
bipartition $(\lambda^{1},\lambda^{2})$. We now define a new bipartition $(%
\widetilde{\lambda}^{1},\widetilde{\lambda}^{2})$ from its symbol $(\binom{%
\widetilde{L}_{2}}{\widetilde{L}_{1}}$ as follows.

\noindent Suppose first $s_{2}\geq s_{1}.\;$Consider $x_{1}=\min\{t\in
L_{1}\}.\;$We associate to $x_{1}$ the integer $y_{1}\in L_{2}$ such that 
\begin{equation}
y_{1}=\left\{ 
\begin{array}{l}
\max\{z\in L_{2}\mid z\leq x_{1}\}\text{ if }\min\{z\in L_{2}\}\leq x_{1},
\\ 
\max\{z\in L_{2}\}\text{ otherwise.}%
\end{array}
\right.   \label{algo1}
\end{equation}
We repeat the same procedure to the lines $L_{2}-\{y_{1}\}$ and $%
L_{1}-\{x_{1}\}.$ By induction this yields a sequence $%
\{y_{1},...,y_{d+1+s_{1}-s_{2}}\}\subset L_{2}.\;$Then we define $\widetilde{%
L}_{1}$ as the line obtained by reordering the integers of $%
\{y_{1},...,y_{d+1+s_{2}-s_{1}}\}$ and $\widetilde{L}_{2}$ as the line
obtained by reordering the integers of $L_{2}-%
\{y_{1},...,y_{d+1+s_{1}-s_{2}}\}+L_{1}$ (i.e. by reordering the set
obtained by replacing in $L_{2}$ the entries $y_{1},...,y_{d+1+s_{1}-s_{2}}$
by those of $L_{1}$).

\noindent Now, suppose $s_{2}<s_{1}.\;$Consider $x_{1}=\min\{t\in L_{2}\}.\;$%
We associate to $x_{1}$ the integer $y_{1}\in L_{1}$ such that 
\begin{equation}
y_{1}=\left\{ 
\begin{array}{l}
\min\{z\in L_{1}\mid x_{1}\leq z\}\text{ if }\max\{z\in L_{1}\}\geq x_{1},
\\ 
\min\{z\in L_{1}\}\text{ otherwise.}%
\end{array}
\right.   \label{algo2}
\end{equation}
We repeat the same procedure to the lines $L_{1}-\{y_{1}\}$ and $%
L_{2}-\{x_{1}\}$ and obtain a sequence $\{y_{1},...,y_{d+1+s_{1}-s_{2}}\}%
\subset C_{1}.\;$Then we define $\widetilde{L}_{2}$ as the line obtained by
reordering the integers of $\{y_{1},...,y_{d+1+s_{2}-s_{1}}\}$ and $%
\widetilde{L}_{1}$ as the line obtained by reordering the integers of $%
L_{1}-\{y_{1},...,y_{d+1+s_{2}-s_{1}}\}+L_{2}.$

\begin{example}
\label{ex1} Assume $(s_{1},s_{2})=(0,3)$ and consider the bipartition of $38$
given by $(\lambda^{1},\lambda^{2})=(6.5.5.4,5.5.3.3.2)$. Then $d=7$ and:%
{\small 
\begin{align*}
S(\lambda^{1},\lambda^{2}) & =%
\begin{tabular}{|l|l|l|l|l|lll}
\hline
$-4+0$ & $-3+0$ & $-2+0$ & $-1+2$ & $0+3$ & $1+3$ & \multicolumn{1}{|l}{$2+5$%
} & \multicolumn{1}{|l|}{$3+5$} \\ \hline
$-4+0$ & $-3+4$ & $-2+5$ & $-1+5$ & $0+6$ &  &  &  \\ \cline{1-5}
\end{tabular}
\\
S(\lambda^{1},\lambda^{2}) & =%
\begin{tabular}{|c|c|c|c|c|ccc}
\hline
$-4$ & $-3$ & $-2$ & $1$ & $3$ & $4$ & \multicolumn{1}{|c}{$7$} & 
\multicolumn{1}{|c|}{$8$} \\ \hline
$-4$ & $1$ & $3$ & $4$ & $6$ &  &  &  \\ \cline{1-5}
\end{tabular}%
\end{align*}
} We get $\{y_{1},...,y_{5}\}=\{-4,1,3,4,-2\}$. This gives {\small 
\begin{equation*}
S(\widetilde{\lambda}^{1},\widetilde{\lambda}^{2})=%
\begin{tabular}{|c|c|c|c|c|ccc}
\hline
$-4$ & $-3$ & $1$ & $3$ & $4$ & $6$ & \multicolumn{1}{|c}{$7$} & 
\multicolumn{1}{|c|}{$8$} \\ \hline
$-4$ & $-2$ & $1$ & $3$ & $4$ &  &  &  \\ \cline{1-5}
\end{tabular}
\end{equation*}
} and finally $(\widetilde{\lambda}^{1},\widetilde{\lambda}%
^{2})=(4.4.3.1,5.5.5.4.4.3)$. Observe that both $(\lambda^{1},\lambda^{2})$
and $(\widetilde{\lambda}^{1},\widetilde{\lambda}^{2})$ have rank equal to $%
38$.
\end{example}

\begin{remark}
The previous combinatorial procedure is very closed from that used in \cite%
{JL} corresponding to the combinatorial $R$-matrix 
\begin{equation*}
\mathcal{G}_{\infty,(s_{1},s_{2})}=\mathcal{G}_{\infty,s_{1}}\otimes 
\mathcal{G}_{\infty,s_{2}}\simeq\mathcal{G}_{\infty,s_{2}}\otimes \mathcal{G}%
_{\infty,s_{1}}=\mathcal{G}_{\infty,(s_{2},s_{1})}. 
\end{equation*}
The only difference is that we do not modify here the original multicharge $%
(s_{1},s_{2})$. This means the above $R$-matrix is the map 
\begin{equation*}
\begin{array}{rccl}
R_{(s_{1},s_{2})}^{\infty} & :\Pi^{2}(n) & \rightarrow & \Pi^{2}(n) \\ 
& (\lambda^{1},\lambda^{2}) & \mapsto & (\widetilde{\lambda}^{2},\widetilde{%
\lambda}^{1}).%
\end{array}
\end{equation*}
We proved in \cite{JL} it is a rank preserving crystal isomorphism. Also $%
R_{(s_{1},s_{2})}^{\infty}$ is the unique $\mathfrak{g}_{\infty}$-crystal
isomorphism between the two Fock spaces because there is no multiplicity
into their decomposition in irreducible components (only in level $2$). We
also have $R_{(s_{1},s_{2})}^{\infty}\circ
R_{(s_{2},s_{1})}^{\infty}=R_{(s_{2},s_{1})}^{\infty}\circ
R_{(s_{1},s_{2})}^{\infty}=id$ and we can write 
\begin{equation*}
(\widetilde{\lambda}^{1},\widetilde{\lambda}^{2})=F\circ
R_{(s_{1},s_{2})}^{\infty}(\lambda^{1},\lambda^{2}) 
\end{equation*}
where $F$ is the flip involution defined on the set of bipartitions by $%
F(\mu^{1},\mu^{2})=(\mu^{2},\mu^{1})$.
\end{remark}

By using the crystal isomorphism introduced in (\ref{iso_perm}) and the
previous remark, we get the following proposition.

\begin{proposition}
\label{uni} Assume that $l=2$ and $e=\infty$. Consider $\mathbf{m}^{+}$ and $%
\mathbf{m}^{-}$ separated by the wall $\mathfrak{m}_{1,2}^{(s_{1},s_{2}),%
\infty}$ such that $m^{+}=(m_{1}^{+},m_{2}^{+})$ with $%
m_{1}^{+}-s_{1}-(m_{2}^{+}-s_{2})>0$ and $m^{-}=(m_{1}^{-},m_{2}^{-})$ with $%
m_{1}^{-}-s_{1}-(m_{2}^{-}-s_{2})<0.$

\begin{enumerate}
\item The map: 
\begin{equation*}
\begin{array}{rccl}
\Phi_{(s_{1},s_{2})}^{\infty} & :\Pi^{2}(n) & \rightarrow & \Pi^{2}(n) \\ 
& (\lambda^{1},\lambda^{2}) & \mapsto & (\widetilde{\lambda}^{1},\widetilde{%
\lambda}^{2}),%
\end{array}
\end{equation*}
defines a $\mathfrak{g}_{\infty}$-crystal isomorphism from $\mathcal{G}%
_{\infty,\mathbf{m}^{+},\mathbf{s}}$ to $\mathcal{G}_{\infty,\mathbf{m}^{-},%
\mathbf{s}}$.

\item We have $\Phi_{(s_{1},s_{2})}^{\infty}=F\circ
R_{(s_{1},s_{2})}^{\infty }$ and $(\Phi_{(s_{1},s_{2})}^{%
\infty})^{-1}=R_{(s_{2},s_{1})}^{\infty}\circ F$ is a $\mathfrak{g}_{\infty}$%
-crystal isomorphism from $\mathcal{G}_{\infty,\mathbf{m}^{-},\mathbf{s}}$
to $\mathcal{G}_{\infty,\mathbf{m}^{+},\mathbf{s}}$.

\item This map $\Phi _{(s_{1},s_{2})}^{\infty }$ is the unique crystal
isomorphism between $\mathcal{G}_{\infty ,\mathbf{m}^{+},\mathbf{s}}$ and $%
\mathcal{G}_{\infty ,\mathbf{m}^{-},\mathbf{s}}$ which preserves the rank of
the bipartitions.\footnote{%
More generally, this also shows there is only one graph isomorphism from $%
\mathcal{G}_{\infty ,\mathbf{m}^{+},\mathbf{s}}$ to $\mathcal{G}_{\infty ,%
\mathbf{m}^{-},\mathbf{s}}$ preserving both the labelling of the arrows and
the rank of the bipartitions. }
\end{enumerate}
\end{proposition}

Assume that $l\in\mathbb{N}$. To construct a crystal isomorphism between two
Fock spaces $\mathcal{G}_{\infty,\mathbf{m},\mathbf{s}}$ and $\mathcal{G}%
_{\infty,\mathbf{m}^{\prime},\mathbf{s}}$ where $\mathbf{m}$ and $\mathbf{m}%
^{\prime}$ are separated by the wall $\mathfrak{m}:=\mathfrak{m}_{i,j}^{%
\mathbf{s},\infty}$ (and only by this wall), we define the application: 
\begin{equation*}
\begin{array}{rccl}
\Phi_{\mathfrak{m}}^{\mathbf{s},\infty} & :\Pi^{l}(n) & \rightarrow & \Pi
^{l}(n) \\ 
& (\lambda^{1},\ldots,\lambda^{l}) & \mapsto & (\mu^{1},\ldots,\mu^{l}),%
\end{array}
\end{equation*}
such that $\mu^{k}=\lambda^{k}$ if $k\neq i,j$ and 
\begin{equation*}
(\mu^{i},\mu^{j})=\left\{ 
\begin{array}{rcl}
\Phi^{{(s_{i},s_{j}),\infty}}(\lambda^{i},\lambda^{j}) & \text{ if } & 
m_{i}-s_{i}-(m_{j}-s_{j})>0, \\ 
(\Phi^{{(s_{i},s_{j}),\infty}})^{-1}(\lambda^{i},\lambda^{j}) & \text{
otherwise.} & 
\end{array}
\right. 
\end{equation*}

\begin{proposition}
\label{Prop_isoFI}$\Phi_{\mathfrak{m}}^{\mathbf{s},\infty}$ defines a $%
\mathfrak{g}_{\infty}$-isomorphism of crystal between $\mathcal{G}_{\infty,%
\mathbf{m},\mathbf{s}}$ and $\mathcal{G}_{\infty,\mathbf{m}^{\prime },%
\mathbf{s}}$.
\end{proposition}

\begin{proof}
Set $\boldsymbol{\delta}=\mathbf{m}-\mathbf{s}$ and $\boldsymbol{\delta }%
^{\prime}=\mathbf{m}^{\prime}-\mathbf{s}$.\ By Proposition \ref{Prop_chamber}%
, since we can move in each chamber without changing the crystal structure
we can assume that $\boldsymbol{\delta}$ and $\boldsymbol{\delta}^{\prime}$
have distinct coordinates.\ We can also assume that $\boldsymbol{m}$ and $%
\boldsymbol{m}^{\prime}$ are very closed to each other but belong to
distinct half-spaces defined by the wall $\mathfrak{m}$. So the coordinates
of $\boldsymbol{\delta}-\boldsymbol{\delta}^{\prime}$ are small and we have $%
\delta_{i}-\delta_{j}>0$ and $\delta_{i}^{\prime}-\delta_{j}^{^{\prime}}<0$
(or $\delta_{i}-\delta_{j}<0$ and $\delta_{i}^{\prime}-\delta_{j}^{\prime}>0$%
).\ We will assume $\delta_{i}-\delta_{j}>0$ and $\delta_{i}^{\prime
}-\delta_{j}^{^{\prime}}<0$, the arguments being analogue when $\delta
_{i}-\delta_{j}<0$ and $\delta_{i}^{\prime}-\delta_{j}^{\prime}>0$.

Let $\sigma$ be the permutations of $\{1,\ldots,l\}$ corresponding to the
decreasing reordering of $\boldsymbol{\delta}$. One can then choose the
coordinates of $\boldsymbol{\delta}-\boldsymbol{\delta}^{\prime}$
sufficiently small so that the permutation of $\{1,\ldots,l\}$ yielding the
decreasing reordering of $\boldsymbol{\delta}^{\prime}$ becomes $%
\sigma^{\prime}=\sigma_{1}\circ(i,j)$. We then have $\Phi_{\mathfrak{m}}^{%
\mathbf{s},\infty }=(\sigma^{\prime})^{-1}\circ
R_{(s_{i},s_{j})}^{\infty}\circ\sigma$ where we also denote by $\sigma$ and $%
\sigma^{\prime}$ the crystal isomorphims of type (\ref{iso_perm}) mapping $%
\mathcal{G}_{\infty,\mathbf{m},\mathbf{s}}$ and $\mathcal{G}_{\infty,\mathbf{%
m}^{\prime},\mathbf{s}}$ on the JMMO structures $\mathcal{G}_{\infty,\sigma(%
\mathbf{s)}}$ and $\mathcal{G}_{\infty ,\sigma^{\prime}(\mathbf{s)}},$
respectively. Therefore, $\Phi_{\mathfrak{m}}^{\mathbf{s},\infty}$ is also a
crystal isomorphism between the Fock spaces $\mathcal{G}_{\infty,\mathbf{m},%
\mathbf{s}}$ and $\mathcal{G}_{\infty ,\mathbf{m}^{\prime},\mathbf{s}}$.
\end{proof}

\bigskip

\label{defbij} We now describe the canonical $\mathfrak{g}_{e}$-isomorphisms
when $e\in\mathbb{N}$ is finite. Consider a wall: 
\begin{equation*}
\mathfrak{m}_{i,j,N}^{\mathbf{s},e}:=\{(m_{1},\ldots,m_{l})\
|s_{i}-m_{i}-(s_{j}-m_{j})=N.e\}. 
\end{equation*}
Assume that $\mathbf{m}$ and $\mathbf{m}^{\prime}$ are separated by this
wall (and only it) and that: 
\begin{equation}
s_{i}-m_{i}-(s_{j}-m_{j})>N.e\text{ and }s_{i}-m_{i}^{\prime}-(s_{j}-m_{j}^{%
\prime})<N.e.   \label{sens}
\end{equation}
To define a $\mathfrak{g}_{e}$-crystal isomorphism between the crystal $%
\mathcal{G}_{e,\mathbf{m},\mathbf{s}}$ and $\mathcal{G}_{e,\mathbf{m}%
^{\prime},\mathbf{s}}$, let us set $\widetilde{\mathbf{s}}=(s_{1},\ldots
,s_{i}-N.e,\ldots,s_{l})$. Then an element $\widetilde{\mathbf{m}}$ is in
the wall $\mathfrak{m}_{i,j,N}^{\mathbf{s},e}$ if and only if $\widetilde {%
\mathbf{m}}$ is in the set 
\begin{equation*}
\mathfrak{m}^{\infty}:=\mathfrak{m}_{i,j}^{\infty,\widetilde{\mathbf{s}}%
}:=\{(m_{1},\ldots,m_{l})\ |(s_{i}-N.e)-m_{i}-(s_{j}-m_{j})=0\}, 
\end{equation*}
where $\widetilde{\mathbf{s}}=(s_{1},s_{2},\ldots,s_{i}-Ne,\ldots,s_{l})$.
The set $\mathfrak{m}_{i,j}^{\infty,\widetilde{\mathbf{s}}}$ is a wall for $%
\widetilde{\mathbf{s}}$ in the case where $e=\infty$ and we known that the
map $\Phi_{\mathfrak{m}^{\infty}}^{\widetilde{\mathbf{s}},\infty}$ of
Proposition \ref{Prop_isoFI} is a $\mathfrak{g}_{\infty}$-crystal
isomorphism between $\mathcal{G}_{\infty,\mathbf{m},\widetilde{\mathbf{s}}}$
and $\mathcal{F}_{\infty,\mathbf{m}^{\prime},\widetilde{\mathbf{s}}}$.

\begin{theorem}
\label{Th_FI}$\Phi_{\mathfrak{m}^{\infty}}^{\widetilde{\mathbf{s}},\infty}$
is a $\mathfrak{g}_{e}$-crystal isomorphism between $\mathcal{G}_{e,\mathbf{m%
},\mathbf{s}}$ and $\mathcal{G}_{e,\mathbf{m}^{\prime},\mathbf{s}}$.
\end{theorem}

\begin{proof}
It follows from Propositions 4.1.1 and 5.2.1 of \cite{JL} that $\Phi _{%
\mathfrak{m}^{\infty}}^{\widetilde{\mathbf{s}},\infty}$ is a $\mathfrak{g}%
_{e}$-crystal isomorphism between $\mathcal{G}_{e,\mathbf{m},\widetilde {%
\mathbf{s}}}$ and $\mathcal{G}_{e,\mathbf{m}^{\prime},\widetilde{\mathbf{s}}}
$.

Now note that $\mathbf{s}$ and $\widetilde{\mathbf{s}}$ are two multicharges
that are equivalent modulo $e$. So we have in fact the equalities $\mathcal{G%
}_{e,\mathbf{m},\widetilde{\mathbf{s}}}=\mathcal{G}_{e,\mathbf{m},{\mathbf{s}%
}}$ and $\mathcal{G}_{e,\mathbf{m}^{\prime},\widetilde{\mathbf{s}}}=\mathcal{%
G}_{e,\mathbf{m}^{\prime},{\mathbf{s}}}$ (see Remark \ref{reegalite}). The
result follows.
\end{proof}

To simplify, we will denote by $\Psi_{\mathbf{m},\mathbf{m}^{\prime}}^{%
\mathbf{s}}$ the above bijection obtained by crossing a single wall when $%
\mathbf{m}$ and $\mathbf{m}^{\prime}$ satisfy (\ref{sens}). The above result
allows to compute certain $\mathfrak{g}_{e}$-crystal isomorphisms between
two arbitrary Fock spaces $\mathcal{F}_{e,\mathbf{m},\mathbf{s}}$ and $%
\mathcal{F}_{e,\mathbf{m}^{\prime},\mathbf{s}}$ by composing the $\mathfrak{g%
}_{\infty}$-crystal isomorphisms of Theorem \ref{Th_FI}.

\subsection{Detecting the highest weight vertices}

\label{hw}

Using the work \cite{JL2}, it is also possible to decide easily whether a
multipartition $\boldsymbol{\lambda}=$ $(\lambda^{1},\ldots,\lambda^{l})$ is
a highest weight vertex in $\mathcal{F}_{e,\mathbf{m},{\mathbf{s}}}$. Such a
criterion is available in the case of the original Uglov structure. Using
our discussion in \S \ref{delta}, we can extend it to the more general
crystal structures we consider here as follows.

Let $\mathbf{m}^{\prime}\in\mathfrak{M}^{\mathbf{s},e}$ and consider the
permutation $\sigma\in\mathfrak{S}_{l}$ and $\mathbf{s}^{\prime}\in \mathbb{Z%
}^{l}$ as defined in \S \ref{delta}.

We define the symbol $S({\boldsymbol{\lambda}})$ of the $l$-partition ${%
\boldsymbol{\lambda}}$ for the multicharge $\mathbf{s}^{\prime}$ by
extending the definition in \S \ref{comb}. This is the $l$-row tableau $S({%
\boldsymbol{\lambda}})$ whose $j$-th row (the rows are numbered from bottom
to top) is : 
\begin{equation*}
\begin{tabular}{|l|l|l|l|l|l|}
\hline
$-u+1+\lambda_{s_{j}^{\prime}+u}^{j}$ & $\cdots$ & $\cdots$ & $\cdots$ & $%
s_{j}^{\prime}-1+\lambda_{2}^{j}$ & $s_{j}^{\prime}+\lambda_{1}^{j}$ \\ 
\hline
\end{tabular}
\end{equation*}
where $u$ is minimal such that $\lambda_{u+s_{i}^{\prime}}^{i}=0$ for all $%
i=1,\ldots,l$.

An $e$-period is a sequence of $e$ boxes $(b_{1},\ldots,b_{e})$ in $S({%
\boldsymbol{\lambda}})$ such that

\begin{itemize}
\item $b_{1}$ contains the greatest entry of $S({\boldsymbol{\lambda}})$,
say $k$,

\item for all $j=1,\ldots,e$ the entry in the box $b_{j}$ is $k-j+1$,

\item if we write $c(b_{j})\in\{1,2,\ldots,l\}$ for the row of the box $b_{j}
$ in $S({\boldsymbol{\lambda}})$ we have%
\begin{equation*}
\sigma^{-1}(b_{1})\geq\sigma^{-1}(b_{2})\geq\cdots\geq\sigma^{-1}(b_{e}). 
\end{equation*}
\end{itemize}

When $S({\boldsymbol{\lambda}})$ admits an $e$-period, one can delete it in $%
S({\boldsymbol{\lambda}})$ and this yields the symbol of a new $l$-partition 
${\boldsymbol{\lambda}}^{\flat}$. We can then apply the same procedure to $S(%
{\boldsymbol{\lambda}}^{\flat})$ providing it also admits an $e$-period.
This process will eventually terminate and then, one of the two following
situations will happen:

\begin{itemize}
\item we finally get a symbol of the empty $l$-partition. In that case ${%
\boldsymbol{\lambda}}$ is a highest weight vertex.

\item we get a symbol without $e$-period which is not a symbol of the empty $%
l$-partition. In that case ${\boldsymbol{\lambda}}$ is not a highest weight
vertex.
\end{itemize}

\begin{example}
Take $e=3$ and $\mathbf{s}=(0,0)$. Recall the notation of \S\ \ref{delta}.

Assume $\mathbf{m}^{\prime}=(m_{1}^{\prime},m_{2}^{\prime})=(1,3)$. We have 
\begin{equation*}
(s_{1}^{\prime},\delta_{1}^{\prime})=(0,1),\ (s_{2}^{\prime},\delta
_{2}^{\prime})=(3,0), 
\end{equation*}
and $\sigma=\text{Id}\in\mathfrak{S}_{2}$. For the $2$-partition ${%
\boldsymbol{\lambda=}}(3.1,2.2.1.1)$, we get the following symbol : 
\begin{equation*}
\begin{tabular}{|c|c|c|ccc}
\hline
$-2$ & $-1$ & $1$ & $2$ & \multicolumn{1}{|c|}{$\boldsymbol{4}$} & 
\multicolumn{1}{|c|}{$\boldsymbol{5}$} \\ \hline
$-2$ & $0$ & $\boldsymbol{3}$ &  &  &  \\ \cline{1-3}
\end{tabular}
. 
\end{equation*}
We see that we have a $3$-period given by the boxes filled by $5$, $4$ and $3
$. By deleting this period, we obtain the symbol 
\begin{equation*}
\begin{tabular}{|c|c|cc}
\hline
$-2$ & $-1$ & $\boldsymbol{1}$ & \multicolumn{1}{|c|}{$\boldsymbol{2}$} \\ 
\hline
$-2$ & $\boldsymbol{0}$ &  &  \\ \cline{1-2}
\end{tabular}
\ \text{ and finally }%
\begin{tabular}{|c|c}
\hline
$-2$ & \multicolumn{1}{|c|}{$-1$} \\ \hline
$-2$ &  \\ \cline{1-1}
\end{tabular}
. 
\end{equation*}
with corresponds to the empty bipartition. Hence ${\boldsymbol{\lambda}}$ is
a highest weight vertex in $\mathcal{G}_{3,(1,3),(0,0)}$.

Now assume $\mathbf{m}^{\prime}=(m_{1}^{\prime},m_{2}^{\prime})=(0,4)$. We
have 
\begin{equation*}
(s_{1}^{\prime},\delta_{1}^{\prime})=(0,0),\ (s_{2}^{\prime},\delta
_{2}^{\prime})=(3,1), 
\end{equation*}
and $\sigma=(1,2)\in\mathfrak{S}_{2}$. For the $2$-partition $(3.1,2.2.1.1)$%
, the symbol that we have to consider is the same as above but the
definition of a period is now twisted by $\sigma$.\ This time, we have no $3$%
-period and ${\boldsymbol{\lambda}}$ is no longer a highest weight vertex in 
$\mathcal{G}_{3,(0,4),(0,0)}$.
\end{example}

\begin{example}
We will here consider the same example as \cite[ex 5.6]{Lo}, 
We assume that $\mathbf{s}=(0,0)$ and $e=2$. We take $n=3$. Then we have $3$
hyperplanes to consider : 
\begin{equation*}
\mathfrak{m}_{1,2,-1}^{(0,0),2}:=\{(m_{1},m_{2})\in\mathbb{Q}^{2}\ |\
m_{2}-m_{1}=-2\}, 
\end{equation*}%
\begin{equation*}
\mathfrak{m}_{1,2,0}^{(0,0),2}:=\{(m_{1},m_{2})\in\mathbb{Q}^{2}\ |\
m_{2}-m_{1}=0\}, 
\end{equation*}%
\begin{equation*}
\mathfrak{m}_{1,2,1}^{(0,0),2}:=\{(m_{1},m_{2})\in\mathbb{Q}^{2}\ |\
m_{2}-m_{1}=2\}. 
\end{equation*}
The $2$-partitions of $3$ are : 
\begin{equation*}
(\emptyset,1.1.1),(\emptyset,2.1), (\emptyset,3), (1,1.1), (1,2), (1.1,1),
(1.1.1,\emptyset), (2,1),(2.1,\emptyset),(3,\emptyset) 
\end{equation*}
Now we pick one parameter in each of the four associated chambers :

\begin{enumerate}
\item Take $\mathbf{m}[1]=(m_{1},m_{2})\in\mathbb{Q}^{2}$ such that $%
m_{2}-m_{1}<-2$. 

\item Take $\mathbf{m}[2]=(m_{1},m_{2})\in\mathbb{Q}^{2}$ such that $%
0>m_{2}-m_{1}>-2$.

\item Take $\mathbf{m}[3]=(m_{1},m_{2})\in\mathbb{Q}^{2}$ such that $%
2>m_{2}-m_{1}>0$.

\item Take $\mathbf{m}[4]=(m_{1},m_{2})\in\mathbb{Q}^{2}$ such that $%
m_{2}-m_{1}>2$.
\end{enumerate}

The following table gives the isomorphisms computed using our procedure, the
notation $(\star)$ indicates that the associated bipartitions are highest
weight vertices for the parameters considered. 
\begin{equation*}
\begin{array}{|c|c|c|c|c|}
\hline
& \mathbf{m}[1] & \mathbf{m}[2] & \mathbf{m}[3] & \mathbf{m}[4] \\ 
\hline\hline
(\star) & (\emptyset,1.1.1) & (\emptyset,1.1.1) & (1.1.1,\emptyset) & 
(1.1.1,\emptyset) \\ \hline
& (\emptyset,2.1) & (\emptyset,2.1) & (2.1,\emptyset) & (2.1,\emptyset ) \\ 
\hline
(\star) & (\emptyset,3) & (1.1.1,\emptyset) & (\emptyset,1.1.1) & 
(3,\emptyset) \\ \hline
& (1,1.1) & (1,1.1) & (1.1,1) & (1.1,1) \\ \hline
& (1,2) & (\emptyset,3) & (3,\emptyset) & (2,1) \\ \hline
(\star) & (1.1,1) & (1,2) & (2,1) & (1,1.1) \\ \hline
& (1.1.1,\emptyset) & (1.1,1) & (1,1.1) & (\emptyset,1.1.1) \\ \hline
& (2,1) & (2,1) & (1,2) & (1,2) \\ \hline
& (2.1,\emptyset) & (2.1,\emptyset) & (\emptyset,2.1) & (\emptyset ,2.1) \\ 
\hline
& (3,\emptyset) & (3,\emptyset) & (\emptyset,3) & (\emptyset,3) \\ \hline
\end{array}
\end{equation*}
\end{example}

\section{Wall crossing bijections for Cherednik algebras}

\label{Sec_WCB}We now give an interpretation of the above isomorphisms in
the context of rational Cherednik algebras. In \cite{Lo}, Losev has
introduced certain combinatorial maps between the sets of simple modules in
the category $\mathcal{O}$ for rational Cherednik algebras. We are going to
explain how these maps (called \textquotedblleft wall-crossing
bijections\textquotedblright) are connected with our crystal isomorphisms.
We refer to \cite{Losurvey} for more details on the representation theory of
Cherednik algebras and for problems we are interested in this section.

\subsection{Rational Cherednik algebras}

Let $\mathcal{H}_{\kappa,\mathbf{s}}(n)$ be the rational Cherednik algebra
associated with the complex reflection group of type $W:=G(l,1,n)$ acting on 
$\mathfrak{h}:=\mathbb{C}^{n}$. As a vector space, this algebra is $S(%
\mathfrak{h}^{\ast})\otimes\mathbb{C}W\otimes S(\mathfrak{h})$ (where $S(V) $
denotes the symmetric algebra of the vector space $V$).\ It admits a
presentation by generators and relations for which we refer to \cite[\S 2.3]%
{GoLo}.

Importantly, this presentation depends on a parameter $s:=(\kappa ,\mathbf{s}%
)\in\mathbb{C}\times\mathbb{C}^{l}$. This parameter is the one used in \cite%
{Lo,Losurvey} as well as in \cite{GoLo} (the reader may look at the
relations between the different parametrizations given in \cite[\S 2.3.2]%
{GoLo})

We will consider the category $\mathcal{O}_{n,s}$ for this algebra whose
simple objects are parametrized by the set $\Pi^{l}(n)$, which also index
the set of irreducible representations of the complex reflection group $W$
in characteristic $0$.

\begin{remark}
An important problem in this theory is to compute the support of a simple
module in the category $\mathcal{O}_{n,s}$ parametrized by an $l$-partition $%
{\boldsymbol{\lambda}}$. This in particular leads to a classification of the
finite dimensional simple modules in this category.
\end{remark}

As explained in \cite[\S 4]{Lo} and \cite[\S 3.1.3]{Losurvey}, in most of
the questions relative to the study of the category $\mathcal{O}$ we can
assume (and we will do in the sequel) the following condition is satisfied.

\begin{condition}
In the rest of the paper we assume that:

\begin{enumerate}
\item $\kappa=\frac{r}{e}$ is a positive rational number where $r$ and $e$
are relatively prime,

\item $r.s_{j}\in\mathbb{Z}$ for any $j=1,\ldots,l$.
\end{enumerate}
\end{condition}

In particular, we have now $s:=(\kappa,\mathbf{s})\in\mathbb{Q}\times 
\mathbb{Q}^{l}$. Let us denote by $\mathcal{S}$ the subset of $\mathbb{Q}%
\times\mathbb{Q}^{l}$ satisfying $(1)$ and $(2)$.

In \cite{Sh}, Shan has introduced an action of the quantum group $\mathfrak{g%
}_{e}$ on a Fock space defined from a categorical action on the direct sum
over $n$ of the categories $\mathcal{O}_{n,s}$. This action heavily depends
on the choice of the parameter $s$. It also induces a structure of $%
\mathfrak{g}_{e}$-crystal on the set of $l$-partitions. This crystal
structure can be defined by using a relevant total order $\leq$ on $z$-nodes
as we did in Section \ref{Gr}. Consider $\gamma=(a,b,c)$ and $\gamma^{\prime
}=(a^{\prime},b^{\prime},c^{\prime})$ two such $z$-nodes.

\begin{definition}
We set $\gamma\leq\gamma^{\prime}$ if and only if 
\begin{equation*}
\kappa l(b-a+s_{c})-c\leq\kappa
l(b^{\prime}-a^{\prime}+s_{c^{\prime}})-c^{\prime}. 
\end{equation*}
The associated oriented graph is denoted by $\mathcal{G}_{s}$ with indexing
set $I_{s}$.
\end{definition}

Observe this is indeed a total order since the equality 
\begin{equation*}
\kappa l(b-a+s_{c})-c=\kappa l(b^{\prime }-a^{\prime }+s_{c^{\prime
}})-c^{\prime },
\end{equation*}%
implies that $l$ divides $c-c^{\prime }$. But $0\leq \left\vert c-c^{\prime
}\right\vert <l$ so we have in fact $c=c^{\prime }$ and $\gamma =\gamma
^{\prime }$. We are going to see that the graph structure $\mathcal{G}_{s}$
coincides with a graph structure already defined which is a crystal
structure up to reparametrization of the colors. Since $r$ and $e$ are
relatively prime, we have $e\mathbb{Z}+r\mathbb{Z}=\mathbb{Z}$ and $e\mathbb{%
Z}\cap r\mathbb{Z}=er\mathbb{Z}$.\ Then, for any integer $a$ there exists a
unique pair $(c,d)$ such that $a=ed+rc$ and $c\in \{0,\ldots ,e-1\}$.\ Since
we have $rs_{j}\in \mathbb{Z}$ for any $j=1,\ldots ,l$, we can set: 
\begin{equation}
rs_{j}=ed_{j}+rc_{j},  \label{defc}
\end{equation}%
where $c_{j}\in \{0,\ldots ,e-1\}$. Note then that $c_{j}$ is equivalent to $%
s_{j}$ modulo $\kappa ^{-1}\mathbb{Z}$. Recall Definition \ref{DefIs}. We
have the following elementary lemma.

\begin{lemma}
The map 
\begin{equation*}
\begin{array}{rccl}
\psi : & \mathbb{Z}/e\mathbb{Z} & \rightarrow  & I_{s} \\ 
& i(\text{mod }e) & \mapsto  & i(\text{mod }\kappa ^{-1}\mathbb{Z})%
\end{array}%
\end{equation*}%
is well defined and is a bijection 
\end{lemma}

\begin{proof}
Assume $i_{1}$ and $i_{2}$ are such that $i_{1}=i_{2}($mod $e)$ and set $%
i_{1}=i_{2}+ae$ with $a\in \mathbb{Z}$. Then $i_{1}=i_{2}+ar\kappa ^{-1}$
since $\kappa =\frac{r}{e}$ thus $i_{1}=i_{2}($mod $\kappa ^{-1}\mathbb{Z})$
and $\psi $ is well-defined from $\mathbb{Z}/e\mathbb{Z}$ to $\mathbb{Q}%
/\kappa ^{-1}\mathbb{Z}$. Now we can write any $i\in \mathbb{Z}$ on the form 
$i=(i-s_{1})+s_{1}$.\ So $\psi (\mathbb{Z}/e\mathbb{Z})\subset I_{s}$ (see
Definition \ref{DefIs}). Conversely, for any integer $x$ and any $j=1,\ldots
,l$, we have 
\begin{equation*}
x+s_{j}(\text{mod }\kappa ^{-1}\mathbb{Z})=x+c_{j}(\text{mod }\kappa ^{-1}%
\mathbb{Z})
\end{equation*}%
so that $\psi (x+c_{j}($mod $e))=x+s_{j}($mod $\kappa ^{-1}\mathbb{Z})$ and $%
\psi $ is surjective. Finally assume $\psi (i_{1}($mod $e))=\psi (i_{2}($mod 
$e))$.\ This implies that $i_{1}-i_{2}$ belongs to $\kappa ^{-1}\mathbb{Z=}%
\frac{e}{r}\mathbb{Z}$. So $i_{1}-i_{2}$ belongs to $e\mathbb{Z}$ because $e$
and $r$ are relatively prime and $i_{1}-i_{2}$ is an integer.
\end{proof}

For $j=1,\ldots ,l$ write also%
\begin{equation*}
m_{j}=s_{j}-\frac{j}{\kappa l}=s_{j}-\frac{je}{rl}\in \mathbb{Q},
\end{equation*}%
and set $\boldsymbol{m}=(m_{1},\ldots ,m_{l})\in \mathbb{Q}^{l}$.

\subsection{Relation with extended JMMO Fock space structure (2)}

In the following proposition, we will consider the oriented graph $\mathcal{G%
}_{s}$ with indexing set $I_{s}$ and the crystal graph $\mathcal{G}_{e,%
\mathbf{m},\mathbf{c}}$ which indexing set is $\mathbb{Z}/e\mathbb{Z}$.

\begin{proposition}
\label{Prop_Iso}The colored oriented graph $\mathcal{G}_{s}$ is equivalent
to the crystal $\mathcal{G}_{e,\mathbf{m},\mathbf{c}}$. More precisely, we
have an isomorphism of oriented graphs $\theta$ between $\mathcal{G}_{e,%
\mathbf{m},\mathbf{c}}$ and $\mathcal{G}_{s}$ and according to the notation
of \S \ref{Gr}, the following maps : 
\begin{equation*}
\begin{array}{rccl}
\Psi & \Pi^{l}(n) & \rightarrow & \Pi^{l}(n) \\ 
& {\boldsymbol{\lambda}} & \mapsto & {\boldsymbol{\lambda}}%
\end{array}
,\ 
\begin{array}{rccl}
\psi & \mathbb{Z}/e\mathbb{Z} & \rightarrow & I_{s} \\ 
& i(\text{mod }e) & \mapsto & i(\text{mod }\kappa^{-1}\mathbb{Z}).%
\end{array}
\end{equation*}
\end{proposition}

\begin{proof}
First, we show that two boxes have the same residue for $(\kappa,\mathbf{s})$
if and only they have the same residue for $(1/e,\mathbf{c})$. This follows
from the fact that, for two nodes $(x,y,i)$ and $(x^{\prime},y^{\prime},j)$
of an $l$-partition, we have 
\begin{equation*}
x-y+s_{i}=x^{\prime}-y^{\prime}+s_{j}+\kappa^{-1}\mathbb{Z}
\end{equation*}
if and only if 
\begin{equation*}
r(x-y+s_{i})=r(x^{\prime}+y^{\prime}+s_{j})+e\mathbb{Z}, 
\end{equation*}
because $\kappa=\frac{r}{e}$. By using (\ref{defc}) we get%
\begin{equation*}
r(x-y+c_{i})=r(x^{\prime}+y^{\prime}+c_{j})+e\mathbb{Z}. 
\end{equation*}
Now both $(x-y+c_{i})$ and $(x^{\prime}+y^{\prime}+c_{j})$ are integers. As $%
(e,r)=1$, we have 
\begin{equation*}
x-y+c_{i}=x^{\prime}+y^{\prime}+c_{j}+e\mathbb{Z}, 
\end{equation*}
which is what we wanted to show. Observe also that the residue of $(x,y,i)$
for $(1/e,\mathbf{c})$ and $(\kappa,\mathbf{s})$ are respectively equal to 
\begin{equation*}
(x-y+c_{i})(\mathrm{mod}e)\text{ and }(x-y+s_{i})(\mathrm{mod}\kappa ^{-1}%
\mathbb{Z}), 
\end{equation*}
and we have 
\begin{equation*}
(x-y+c_{i})\equiv(x-y+s_{i})(\mathrm{mod}\kappa^{-1}\mathbb{Z}). 
\end{equation*}

It just remains to show that the order $\leq$ on $z$-nodes corresponds to
the order $\preceq_{\mathbf{m}}$.

Let $\gamma=(x,y,i)$ and $\gamma^{\prime}=(x^{\prime},y^{\prime},j)$ be two
nodes of the $l$-partition. Then we have the equivalences%
\begin{multline*}
\gamma\leq\gamma^{\prime}\Longleftrightarrow\kappa(x-y+s_{c})-\frac{c}{l}%
=\kappa(x^{\prime}-y^{\prime}+s_{c^{\prime}})-\frac{c^{\prime}}{l} \\
\Longleftrightarrow x-y+m_{i}\leq x^{\prime}-y^{\prime}+m_{j}.
\end{multline*}
which is exactly what we wanted.
\end{proof}

\begin{remark}
Combining this proposition with the results in \S \ref{hw}, we obtain a
criterion to decide whether an $l$-partition is a highest weight in $%
\mathcal{G}_{\mathbf{s}}$. The finite dimensional modules in the associated
category $\mathcal{O}$ of the Cherednik algebra are in particular
parametrized by $l$-partitions which satisfy this criterion. We refer to 
\cite{GHe} for other results in this direction and a detailed investigation
of the crystal action of the Heisenberg algebra on $\mathcal{G}_{\mathbf{s}}$
which also appears in the study of this problem.
\end{remark}

Let us now review the wall crossing bijections. Fix $n\in\mathbb{N}$. In the
set of parameters $\mathcal{S}$, one can define certain hyperplanes called
essential walls. Set%
\begin{equation*}
h_{j}=\kappa m_{j}=\kappa s_{j}-j/l,\ j=1,\ldots,l. 
\end{equation*}

\begin{definition}
Recall $n$ is fixed. Given $i,j$ distinct in $\{1,\ldots,l\}$, the essential
wall parametrized by $(i,j)$ is the set of parameters $s=(\kappa ,\mathbf{s}%
)\in\mathcal{S}$ such that there exists an integer $a$ satisfying:

\begin{enumerate}
\item $\left\vert a\right\vert <n,$

\item $m_{i}-m_{j}=a,$

\item $s_{i}-s_{j}-a\in\kappa^{-1}\mathbb{Z}$.
\end{enumerate}
\end{definition}

Consider two parameters $s=(\kappa,\mathbf{s})$ and $s^{\prime}=(\kappa
^{\prime},\mathbf{s}^{\prime})$ in $\mathcal{S}$ satisfying : 
\begin{equation}
\kappa-\kappa^{\prime}\in\mathbb{Z}\text{ and }\forall j\in\{1,\ldots ,l\}\
\kappa s_{j}-\kappa^{\prime}s_{j}^{\prime}\in\mathbb{Z} 
\label{ParamIsoLosev}
\end{equation}
and such that $s$ and $s^{\prime}$ are separated by one essential wall.

\begin{remark}
Assume $s$ and $s^{\prime}$ satisfy (\ref{ParamIsoLosev}). This means that
there exist $k\in\mathbb{Z}$ and $t\in\mathbb{Z}$ such that $%
r^{\prime}s_{j}^{\prime}=rs_{j}+ke$ and $r^{\prime}=r+t.e$.

By (\ref{defc}), for each $j=1,\ldots,l$, there exists a unique $%
(c_{j},c_{j}^{\prime})\in\{0,\ldots,e-1\}^{2}$ and $(d_{j},d_{j}^{\prime})\in%
\mathbb{Z}^{2}$ such that 
\begin{equation*}
rs_{j}=ed_{j}+rc_{j},\ r^{\prime}s_{j}^{\prime}=ed_{j}^{\prime}+r^{\prime
}c_{j}^{\prime}. 
\end{equation*}
We obtain: 
\begin{equation*}
\begin{array}{rcl}
r^{\prime}s_{j}^{\prime} & = & rs_{j}+ke \\ 
& = & ed_{j}+rc_{j}+ke \\ 
& = & ed_{j}+r^{\prime}c_{j}-tec_{j}+ke \\ 
& = & e(d_{j}-te+k)+r^{\prime}c_{j}.%
\end{array}
\end{equation*}
We thus have $c_{j}^{\prime}=c_{j}$ and $d_{j}^{\prime}=d_{j}-te+k$.

By Proposition \ref{Prop_Iso} $\mathcal{G}_{s}$ and $\mathcal{G}%
_{s^{^{\prime }}}$ are respectively isomorphic to $\mathcal{G}_{e,\mathbf{m},%
\mathbf{c}}$ and $\mathcal{G}_{e,\mathbf{m}^{\prime},\mathbf{c}}$ for good
choices of $\mathbf{m}$ and $\mathbf{m}^{\prime}$. But $\mathcal{G}_{e,%
\mathbf{m}^{\prime},\mathbf{c}}$ and $\mathcal{G}_{e,\mathbf{m},\mathbf{c}}$
are isomorphic as crystal graphs because the associated Fock spaces have the
same multicharge. So $\mathcal{G}_{\mathbf{s}}$ and $\mathcal{G}_{\mathbf{s}%
^{\prime}}$ are isomorphic as soon as $s$ and $s^{\prime}$ satisfy (\ref%
{ParamIsoLosev}).
\end{remark}

In \cite[Prop. 5.9]{Lo}, Losev defines a bijection $\mathfrak{wc}%
_{s\rightarrow s^{\prime}}$ between $l$-partitions which is an isomorphism
of graphs in the sense of \S \ref{Giso}.

\begin{proposition}
\label{Prop_bijecLOsev}Given $s$ and $s^{\prime}$ in $\mathcal{S}$
satisfying (\ref{ParamIsoLosev}) and separated by the essential wall
parametrized by $(i,j)$, there is a bijection $\mathfrak{wc}_{s\rightarrow
s^{\prime}}$ on the set of $l$-partitions such that if ${\boldsymbol{\mu}}=%
\mathfrak{wc}_{s\rightarrow s^{\prime}}({\boldsymbol{\lambda}})$ we have

\begin{itemize}
\item $\lambda^{k}=\mu^{k}$ if $k\neq i$ and $k\neq j$,

\item $(\mu^{i},\mu^{j})=\mathfrak{wc}_{(\kappa,s_{i},s_{j})\rightarrow
(\kappa^{\prime},s_{i}^{\prime},s_{j}^{\prime})}(\lambda^{i},\lambda^{j})$
\end{itemize}

where $\mathfrak{wc}_{(\kappa ,s_{i},s_{j})\rightarrow (\kappa ^{\prime
},s_{i}^{\prime },s_{j}^{\prime })}$ is the unique graph isomorphism between 
$\mathcal{G}_{s}$ and $\mathcal{G}_{s^{\prime }}$ preserving both the
labelling of the arrows and the rank of the bipartitions.
\end{proposition}

Let us look at the essential hyperplanes and check that they are the same as
the ones define in \S \ref{act2}. Assume that we are in an essential
hyperplane. This implies that there exists $a\in\mathbb{Z}$ such that 
\begin{equation*}
s_{i}-s_{j}-a\in\kappa^{-1}\mathbb{Z}\qquad\text{ and }\qquad
m_{i}-m_{j}=e.a. 
\end{equation*}
Thus there exists $x\in\mathbb{Z}$ such that $s_{i}-s_{j}=a+\kappa^{-1}x$.
We obtain: 
\begin{equation*}
e(d_{i}-d_{j}-x)=r(c_{j}-c_{i}+a). 
\end{equation*}
This implies that $c_{j}-c_{i}=a+e\mathbb{Z}$ because $e$ and $r$ are
relatively prime. Thus, we have : 
\begin{equation*}
c_{j}-m_{j}-(c_{i}-m_{i})\in e\mathbb{Z}. 
\end{equation*}
Conversely, if we have that $c_{j}-m_{j}-(c_{i}-m_{i})=eN$ with $N\in 
\mathbb{Z}$ with the above relation between the parameters. We obtain that 
\begin{equation*}
m_{i}-m_{j}=c_{i}-c_{j}+eN. 
\end{equation*}
Let us then set $a=c_{i}-c_{j}+eN$ which is in $\mathbb{Z}$. Now by (\ref%
{defc}) we have $s_{i}=\kappa^{-1}d_{i}+c_{i}$ and $s_{j}=\kappa
^{-1}d_{j}+c_{j}$. We thus get%
\begin{equation*}
s_{i}-s_{j}-a=\kappa^{-1}(d_{i}-d_{j})+(c_{i}-c_{j})-a=%
\kappa^{-1}(d_{i}-d_{j}+rN)\in\kappa^{-1}\mathbb{Z}. 
\end{equation*}

We now want to check that the wall crossing bijections correspond to our
bijections defined in \S \ref{defbij}.

\begin{theorem}
The wall crossing bijection $\mathfrak{wc}_{s\to s^{\prime}}$ corresponds to 
$\Psi^{\mathbf{c}}_{\mathbf{m},\mathbf{m}^{\prime}}$.
\end{theorem}

\begin{proof}
Since both $\mathfrak{wc}_{s\rightarrow s^{\prime }}$ and $\Psi _{\mathbf{m},%
\mathbf{m}^{\prime }}^{\mathbf{c}}$ modify the component of indices $i$ and $%
j$ in the $l$-partition it suffices to show that for any bipartition $%
(\lambda ^{1},\lambda ^{2})$ we have $\mathfrak{wc}_{(\kappa
,s_{i},s_{j})\rightarrow (\kappa ^{\prime },s_{i}^{\prime },s_{j}^{\prime
})}(\lambda ^{i},\lambda ^{j})=\Psi _{(m_{i},m_{^{j}}),(m_{i}^{\prime
},m_{j}^{\prime })}^{(c_{i},c_{j})}(\lambda ^{i},\lambda ^{j})$. But $%
\mathfrak{wc}_{(\kappa ,s_{i},s_{j})\rightarrow (\kappa ^{\prime
},s_{i}^{\prime },s_{j}^{\prime })}$ and $\Psi
_{(m_{i},m_{^{j}}),(m_{i}^{\prime },m_{j}^{\prime })}^{(c_{i},c_{j})}$ are
graph isomorphisms which preserve the labelling of the arrows and the rank
of the bipartitions. Since there is only one such isomorphism, they coincide.
\end{proof}

\bigskip

\begin{remark}
Assume now $s:=(\kappa,\mathbf{s})$ where $\kappa$ is a rational negative
number. To each $l$-partition ${\boldsymbol{\lambda}}=(\lambda^{1},\ldots,%
\lambda^{l})$ we associate its conjugate ${\boldsymbol{\lambda}}%
^{\#}=((\lambda^{l})^{\#},\ldots,(\lambda^{1})^{\#})$ where $(\lambda
^{c})^{\#}$ is the conjugate of the partition $\lambda^{c}$ for any $%
k=1,\ldots,l.$ There is a natural bijection between the nodes of ${%
\boldsymbol{\lambda}}$ and ${\boldsymbol{\lambda}}^{\#}$ which maps $%
\gamma=(a,b,c)\in{\boldsymbol{\lambda}}$ on $\gamma^{\#}=(b,a,l-c)$. Set $%
s^{\#}:=(-\kappa,\mathbf{s}^{\#})$ where $\mathbf{s}^{\#}=(-s_{l},%
\ldots,-s_{1})$. The map $\#$ then defines an anti-isomorphism from the
graphs $\mathcal{G}_{s}$ to $\mathcal{G}_{s^{\#}}$ which sends each $l$%
-partition ${\boldsymbol{\lambda}}$ on ${\boldsymbol{\lambda}}^{\#}$ and
each arrow ${\boldsymbol{\lambda}}\overset{z}{\rightarrow}{\boldsymbol{\mu}}$
on ${\boldsymbol{\lambda}}^{\#}\overset{-z}{\rightarrow}{\boldsymbol{\mu}}%
^{\#}$.\ This implies that the wall crossing maps for $\kappa$ and $-\kappa$
coincide up to conjugation by $\#$.
\end{remark}

\bigskip 



\end{document}